\documentclass[11pt]{amsart}
\usepackage{amssymb}
\usepackage{amsthm}
\usepackage{amsmath}
\usepackage[T1]{fontenc}
\usepackage[utf8]{inputenc}
\def\bC{\mathbb C}
\def\bR{\mathbb R}
\def\bN{\mathbb N}
\def\cR{\mathcal R}
\def\cK{\mathcal K}
\def\cF{\mathcal F}
\def\e{\varepsilon}
\def\wto{\overset{w}{\to}}
\def\Re{{\rm Re}\,}

\newtheorem{thm}{Theorem}[section]
\newtheorem{lemma}[thm]{Lemma}

\newtheorem{theorem}[thm]{Theorem}
\newtheorem{proposition}[thm]{Proposition}
\newtheorem{corollary}[thm]{Corollary}
\theoremstyle{definition}
\newtheorem{definition}[thm]{Definition}

\newtheorem{example}[thm]{Example}

\theoremstyle{remark}
\newtheorem{remark}[thm]{Remark}

\theoremstyle{definition}

\theoremstyle{remark}

\title[Numerical Ranges of Real-linear Operators]{Some Generalized Numerical Ranges of Real-linear Operators on Hilbert Spaces}

\author{Arup Majumdar}
\address{Institute of Mathematics, Czech Academy of Sciences,
\v{Z}itn\'a 25, 115~67 Prague 1, Czech Republic}
\email{arupmajumdar93@gmail.com}

\author{Vladimir M\"uller}
\address{Institute of Mathematics, Czech Academy of Sciences,
\v{Z}itn\'a 25, 115~67 Prague 1, Czech Republic}
\email{muller@math.cas.cz}

\begin{document}
\begin{abstract}
    This paper introduces and systematically investigates several classes of numerical ranges associated with real-linear operators on Hilbert spaces, including the  essential numerical range, maximal numerical range, essential maximal numerical range, $q$-numerical range, and trace numerical range. We establish fundamental structural properties of these sets and characterize their geometric shape. In particular, we analyze how the real-linear nature of the operators influences the geometry and spectral behavior of the corresponding numerical ranges, thereby extending and enriching the theory of numerical ranges from complex-linear to real-linear operator settings.
\end{abstract}

 \maketitle



    \section{Introduction}

    The numerical range is one of the fundamental objects in the spectral and operator theory of bounded linear operators. Let $H$ be a complex Hilbert space and let $\mathcal{B}(H)$ denote the algebra of all bounded linear operators on $H$. For $T\in\mathcal{B}(H)$, the numerical range of $T$ is defined by
$W(T)=\{\langle Tx,x\rangle: x\in H, \|x\|=1\}.$
Although this definition is elementary, the set $W(T)$ contains substantial information concerning the geometry, spectral behavior, and norm estimates of the operator. One of the most fundamental results in the subject is the Toeplitz--Hausdorff theorem, which asserts that $W(T)$ is a convex subset of the complex plane for every bounded linear operator $T$. The closure of the numerical range contains the spectrum of $T$, and the numerical radius
$$w(T)=\sup\{|\lambda|:\lambda\in W(T)\}$$
defines a norm on $\mathcal{B}(H)$ equivalent to the operator norm. These properties make the numerical range an indispensable tool in operator theory, matrix analysis, and spectral approximation; see, for example, {\cite{GauWu2021, BonsallDuncan1973}}.

The theory admits several important refinements. Among them are the maximal numerical range, the essential numerical range, and various generalized numerical ranges that encode additional geometric or spectral information. The maximal numerical range is closely related to norm-attaining sequences, while the essential numerical range captures the behavior of an operator modulo compact perturbations. In particular, if $\|T\|_e$ denotes the essential norm of $T$, the essential numerical range can be characterized through weakly null sequences of unit vectors. Such descriptions are especially useful in infinite-dimensional Hilbert spaces, where ordinary norm-attainment and compactness phenomena may fail. The theory of essential and maximal numerical ranges has consequently played an important role in the study of the asymptotic behavior of operators; see \cite{FillmoreStampfliWilliams1972, stampfli1970norm}.

The $q$-numerical range interpolates between several familiar numerical-range constructions.
The $q$-numerical range has been extensively studied, particularly for normal operators, and reveals geometric information that is not visible from the ordinary numerical range alone; see, for example,
{\cite{chien2002davis, chien2005q}}. Such generalized numerical ranges provide a natural framework for investigating operator properties through constrained quadratic forms.

In recent years, there has been increasing interest in extending numerical range theory from linear to real-linear (antilinear) operators, for reference \cite{amara2026numerical, amara2025maximal, Vmuller, Vmuller2}. Let $H$ be a complex Hilbert space. An operator $A:H\to H$ is called antilinear if
$$
A(x+y)=Ax+Ay,
\qquad
A(\alpha x)=\overline{\alpha}Ax,
\qquad
x,y\in H,\quad \alpha\in\mathbb C
$$ and an operator $\Phi:H\to H$ is called real-linear if

$$
\Phi(x+y)=\Phi x+ \Phi y,
\qquad
\Phi(\alpha x)=\alpha \Phi x,
\qquad
x,y\in H,\quad \alpha\in\mathbb{R}.
$$

In other words, the real-linear operator $\Phi$ is linear with respect to the real scalar field $\mathbb{R}$, but it need not be linear with respect to complex scalars.

Antilinear operators arise naturally in several areas of functional analysis and mathematical physics, particularly in connection with conjugations, symmetries, and block operator structures. A conjugation is a particularly important example: it is an isometric antilinear involution. The study of conjugations and, more generally, antilinear operators therefore provides a natural counterpart to the classical theory of linear operators, see \cite{majumdar2026unbounded}.

Consequently, the numerical range of an antilinear operator possesses a circular symmetry which has no direct analogue in the ordinary numerical range of a linear operator. Earlier investigations showed that numerical ranges of antilinear operators exhibit annulus behavior, in sharp contrast with the convexity furnished by the Toeplitz--Hausdorff theorem in the linear setting; see \cite{hur2021numerical}. More recent results have substantially clarified this phenomenon and show that, on a Hilbert space of dimension at least two, the numerical range of an antilinear operator is in fact a disc, and hence is convex \cite{Vmuller, majumdar2026unbounded}. This illustrates that the geometric behavior of antilinear numerical ranges is simultaneously more rigid and fundamentally different from that of their linear counterparts.

The distinction becomes even more pronounced when one considers asymptotic versions of numerical ranges. In the linear theory, the essential numerical range and the maximal numerical range provide two complementary ways of detecting the behavior of an operator along sequences of vectors which escape compact subsets of the Hilbert space or approach the operator norm. For an antilinear operator $A$, analogous constructions lead to new questions. For instance, one may ask how the essential norm of $A$ is reflected in the geometry of its essential numerical range, whether norm-extremizing weakly null sequences determine a natural  essential maximal numerical range, and how these sets are related to spectral subspaces of the positive linear operator $|A|$.

These questions motivate the present work. We introduce and investigate several generalized numerical ranges associated with bounded real-linear (antilinear) operators on complex Hilbert spaces. In particular, we study the essential numerical range, maximal numerical range, essential maximal numerical range, $q$-numerical range, and  trace numerical range in the real-linear (antilinear) setting. Our aim is not merely to reproduce the corresponding linear definitions, but rather to determine which structural and geometric properties survive the transition to real-linear (antilinear) operators and which phenomena are genuinely new.

A central role in our analysis is played by the interaction between the antilinear structure of $A$ and the associated positive linear operator $|A|=(A^*A)^{\frac{1}{2}}$, where $A^*$ denotes the antilinear adjoint of $A$. Since the product of two antilinear operators is linear, the operator $A^*A$ belongs to $\mathcal{B}(H)$ and is positive. This allows spectral methods for linear positive operators to be combined with the geometry of antilinear numerical ranges. In particular, spectral projections of $|A|$ near the essential norm provide a natural mechanism for describing norm-extremizing sequences and for analyzing essential maximal numerical ranges.

We also investigate the extent to which classical relationships between different numerical ranges remain valid. The linear theory suggests a number of natural inclusions and equality statements, but the antilinear setting exhibits substantial differences. In particular, the circular symmetry forced by antilinearity and the interaction between the quadratic form $\langle Ax,x\rangle$ and the norm $\|Ax\|$ lead to phenomena that have no direct counterpart for linear operators. These differences become especially visible in infinite-dimensional spaces, where weakly null sequences and essential spectral subspaces enter naturally.

The results obtained in this paper provide a systematic framework for these generalized numerical ranges and clarify their geometric structure. In particular, we establish several inclusion relations and structural properties and determine the precise shape of the corresponding sets under appropriate assumptions. We also exhibit examples demonstrating that some natural analogies with the linear theory fail in the real-linear (antilinear) setting. These examples are essential in identifying the correct hypotheses under which the linear intuition can be transferred to real-linear (antilinear) operators.

The paper is organized as follows. Section~2 introduces the fundamental notations and terminology that will be used throughout the paper. Section~3 reviews the classical properties of the essential numerical range of real-linear operators. Section~4 is devoted to the maximal numerical range of antilinear operators and its associated geometric properties. In Section~5, we investigate the essential maximal numerical range of real-linear operators and, using its properties, establish the connectedness of the maximal numerical range of real-linear operators. Section~6 focuses on the $q$-numerical range of antilinear operators and studies its geometric structure. Finally, Section~7 examines the properties of the trace numerical range of real-linear operators.

\section{Basic Notations}
Let $H$ be a complex Hilbert space. A mapping $\Phi:H\to H$ is called real-linear operator if it is continuous, $\Phi(x+y)=\Phi x+\Phi y$ and $\Phi(tx)=t\Phi x$ for all $x,y\in H$ and $t\in\bR$. It is well known that a real-linear mapping is continuous if and only if it is bounded. For a real-linear mapping $\Phi:H\to H$ write $\|\Phi\|=\sup\{\|\Phi x\|: x\in H, \|x\|=1\}$.
Denote by $\cR(H)$ the set of all real-linear operators. Then $\cR(H)$ with naturally defined operations is a real Banach algebra.

A real-linear operator is called linear if $\Phi(\lambda x)=\lambda\Phi x$ for all $x\in H$ and $\lambda\in\bC$, and antilinear if $\Phi(\lambda x)=\bar\lambda\Phi x\quad(x\in H, \lambda\in\bC$).

Every real-linear operator $\Phi: H \to H$ admits a unique decomposition $ \Phi= T +A,$
where $T$ is complex-linear and $A$ antilinear operator. More precisely, $T$ and $A$ are defined by 
$$
Tx
=\frac{1}{2}\big(\Phi x-i\Phi(ix)\big),
\qquad
Ax
=\frac{1}{2}\big(\Phi x+i \Phi(ix)\big),
$$ for every $x\in H$.

If $T:H\to H$ is a linear operator, then $T^{*}$ denotes its adjoint, characterized by
$$
\langle Tx,y\rangle=\langle x,T^{*}y\rangle,
\qquad x,y\in H.
$$
For an antilinear operator $A:H\to H$, its antilinear adjoint $A^*$ is characterized by
$$
\langle Ax,y\rangle=\overline{\langle x,A^{*}y\rangle},
\qquad x,y\in H.
$$
Let $\Phi:H\to H$ be a real-linear operator, $\Phi=T+A$ where $T$ is linear and $A$ antilinear operator. 
The adjoint of $\Phi$ is defined by $$\Phi^{*} = T^{*} + A^{*}.$$ 
Then $\Phi^*$ is again a real-linear operator satisfying
$$
\Re\langle \Phi x,y\rangle=\Re\langle x,\Phi^* y\rangle
$$
for all $x,y\in H$.
The adjoint of real-linear operators satisfies the usual conditions: $(\Phi+\Psi)^*=\Phi^*+\Psi^*$ and $(\Phi\Psi)^*=\Psi^*\Phi^*$ and $(\Phi^*)^*=\Phi$ for all $\Phi,\Psi\in\cR(H)$.

An antilinear operator $C:H\to H$ is called a conjugation if $C$ is an isometry, $C^*=C$ and $C^2=I$ (the identity operator on $H$). If $(x_j)_{j\in\Lambda}$ is an orthonormal basis in $H$ then the mapping $\sum_{j}\alpha_jx_j\mapsto\sum_j\bar\alpha_j x_j$ is a conjugation.

If $C$ is a conjugation and $A:H\to H$ an antilinear operator then both $CA$ and $AC$ are linear operators. This simple observation enables to transfer many result for linear operators to the antilinear setting.

The numerical range of a real-linear operator $\Phi:H\to H$ is defined by $W(\Phi)=\{\langle\Phi x,x\rangle:x\in H,\|x\|=1\}$.  By Theorem 6.4 \cite{Vmuller2}, $W(\Phi)$ is always a convex subset of the complex plane.

If $A:H\to H$ is an antilinear operator then we have a more information. If $\dim H\ge 2$ then the numerical range $W(A)$ is always a disc centered at $0$ (either open or closed, possibly degenerated to  the single point $0$).

For a sequence $\{x_n\}$ in $H$, the notation $ x_n\wto x$ means that $x_n$ converges weakly to $x$, that is,
$$
\langle x_n,y\rangle\to\langle x,y\rangle
\qquad\text{for every }y\in H.
$$

By a subspace $M\subset H$ we always mean a closed subset satisfying $x+y\in M$ and $\lambda x\in M$ for all $x,y\in M$ and $\lambda\in\bC$.
For a set $S$, we denote by $\operatorname{conv}(S)$ the convex hull of $S$, that is, the smallest convex set containing $S$.

For a real-linear operator $\Phi\in\cR(H)$ we denote its null space and range by $N(\Phi)$ and $R(\Phi)$, respectively.
If $M$ is a subspace of $H$, then $\Phi|_M$ denotes the restriction of $\Phi$ to $M$. 


A real-linear operator $\Phi:H\to H$ is called compact if $\overline{\{\Phi x:\|x\|\le 1\}}$ is a compact set.
Let $\Phi=T+A$ with $T$ linear and $A$ antilinear operator. Then $\Phi$ is compact if and only if both $T$ and $A$ are compact. Denote by $\cK(H)$ the set of all compact real-linear operators on $H$. Then $\cK(H)$ is a two-sided ideal in the real Banach algebra $\cR(H)$.

If $\Phi\in\cR(H)$ is compact and $(x_n)$ a sequence of vectors in $H$ satisfying $x_n\wto 0$, then $\Phi x_n\to 0$ in the norm. Indeed, write $\Phi=T+A$ with $T$ linear and $A$ antilinear operator. Since $T$ is a compact linear operator, we have $\|Tx_n\|\to 0$. If $C$ is a conjugation on $H$, then $CA$ is a compact linear operator, and so $\|Ax_n\|=\|CAx_n\|\to 0$. Hence $\|\Phi x_n\|\to 0$.

An operator $\Phi\in\cR(H)$ is compact if and only if $\Phi^*$ is compact. Indeed, let $\Phi=T+A$ be the canonical linear/antilinear decomposition and $C:H\to H$ a conjugation. Then $T^*$ and $(AC)^*$ are compact linear operators, and so $A^*=C (CA^*)=C(AC)^*$ is compact. Hence $\Phi^*$ is compact.

For $\Phi\in\cR(H)$ define the essential norm by $\|\Phi\|_e=\inf\{\|\Phi+\Phi\|:\Psi\in\cK(H)\}.$

\section{Essential numerical range}

\begin{definition} Let $\Phi$ be a bounded real-linear operator on an infinite-dimensional Hilbert space $H.$ The essential numerical range $W_e(\Phi)$ of $\Phi$ is the set of all $\lambda\in\bC$ such that there exists a sequence  $(x_n)$ of unit vectors in $H$ satisfying $x_n\wto 0$ and $\langle \Phi x_n,x_n\rangle\to\lambda$.
\end{definition}

\begin{theorem}\label{T2.2}
Let $\Phi$ be a real-linear operator on an infinite-dimensional Hilbert space $H$ and $\lambda\in \bC$. The following statements are equivalent:

(i) $\lambda\in W_e(\Phi)$;

(ii) for every subspace $M\subset H$ of finite codimension and every $\e>0$ there exists a unit vector $y\in M$ such that $|\langle \Phi y,y\rangle-\lambda|<\e$;

(iii) there exists an orthonormal sequence $(x_n)$ in $H$ such $\langle \Phi x_n,x_n\rangle\to\lambda$.


\end{theorem}

\begin{proof}
 (i)$\Rightarrow$(ii): Without loss of generality we may assume that $\|\Phi\|\le 1$.
Let $\{x_n\}$ be a weakly null sequence of unit vectors such that $\langle \Phi x_n,x_n\rangle\to\lambda$.
Let $F=M^\perp$. Then $\dim F<\infty$. Let $e_1,\dots,e_k$ be an orthonormal basis in $F$. 

Find $n_0$ such that $|\langle x_{n_0},e_j\rangle|<\frac{\e}{6k}$ for all $j=1,\dots,k$ and $|\langle \Phi x_{n_0},x_{n_0}\rangle-\lambda|<\e/3$.
Let 
$$
u=x_{n_0}-\sum_{j=1}^k\langle x_{n_0},e_j\rangle e_j.
$$
Then $u\in M$ and $\|u-x_{n_0}\|< \e/6$.
Set $y=\frac{u}{\|u\|}$. Then $y\in M$, $\|y\|=1$ and 
$$
\|y-x_{n_0}\|\le\|y-u\|+\|u-x_{n_0}\|< |1-\|u\||+\frac{\e}{6}< \frac{\e}{3}.
$$
Moreover,
$$
|\langle \Phi y,y\rangle-\lambda|\le
| \langle \Phi y-\Phi{x_{n_{0}}}, y\rangle |+|\langle \Phi x_{n_0},y-x_{n_0}\rangle|+|\langle \Phi x_{n_0},x_{n_0}\rangle-\lambda|< \e.
$$

\smallskip

(ii)$\Rightarrow$(iii):
Using (ii) we construct inductively unit vectors $y_n$ such that $y_n\perp \{y_1,\dots,y_{n-1}\}$ and $|\langle \Phi y_n,y_n\rangle-\lambda|<n^{-1}$. Hence (iii).
\smallskip

(iii)$\Rightarrow$(i): Clear.
\end{proof}

\begin{theorem}
Let $\Phi$ be a real-linear operator on an infinite-dimensional Hilbert space $H$. Then $W_e(\Phi)$ is a non-empty compact convex subset of $\bC$.
\end{theorem}

\begin{proof}
Clearly $W_e(\Phi)$ is closed by definition. For each $\lambda\in W_e(\Phi)$ we have $|\lambda|\le\|\Phi\|$. Hence $W_e(\Phi)$ is a compact set.

Let $\{x_n\}$ be a weakly-null sequence of unit vectors in $H$. Any accumulation point of the sequence $(\langle\Phi x_n,x_n\rangle)$ belongs to $W_e(\Phi)$. Thus $W_e(\Phi)$ is non-empty.

We show that $W_e(\Phi)$ is a convex set. Let $\lambda,\mu\in W_e(\Phi)$ and $t\in[0,1]$.
We show that $t\lambda+(1-t)\mu\in W_e(\Phi)$.

Let $\e>0$ and let $M\subset H$ be a subspace of finite codimension. Write $\Phi=T+A$, where $T$ is a linear and $A$ an antilinear operator.

Find a unit vector $x\in M$ such that $|\langle \Phi x,x\rangle-\lambda|<\e$.
Find a unit vector $y\in M\cap\{x,\Phi x, T^*x,A^*x\}^\perp$ such that $|\langle \Phi y, y\rangle - \mu | < \e$. Set $u=\sqrt{t} x+\sqrt{1-t} y$. Since $y\perp x$, we have $\|u\|=1$. Clearly $u\in M$. We have
$$
\langle \Phi u,u\rangle=t \langle \Phi x,x\rangle +(1-t)\langle \Phi y,y\rangle+\sqrt{t(1-t)}(\langle\Phi x,y\rangle+\langle\Phi y,x\rangle),
$$
where $\langle\Phi x,y\rangle=0$ and $\langle\Phi y,x\rangle=\langle y, T^*x\rangle+\overline{\langle y,A^*x\rangle}=0$.
Thus
$$
\bigl|\langle\Phi u,u\rangle-\bigl(t\lambda+(1-t)\mu\bigr)\bigr|\le
t|\langle\Phi x,x\rangle-\lambda|+(1-t)|\langle\Phi y,y\rangle-\mu|<\e.
$$
Hence $t\lambda+(1-t)\mu\in W_e(\Phi)$ and $W_e(\Phi)$ is a convex subset of $\bC$.
 
\end{proof}

\begin{theorem}\label{T2.4}
Let $\Phi$ be a real-linear operator on an infinite-dimensional Hilbert space $H$. Then 
$$
W_e(\Phi)=\bigcap\{\overline{W(\Phi+\Psi)}: \Psi:H\to H \hbox{ real-linear compact}\}
$$
$$
=
\bigcap\{\overline{W(\Phi+\Psi)}: \Psi:H\to H \hbox{ real-linear of finite rank}\}.
$$
\end{theorem}

\begin{proof}
Let $\lambda\in W_e(\Phi)$ and let $\Psi$ be a real-linear compact operator. Let $\{x_n\}$ be a weakly null sequence such that $\langle\Phi x_n,x_n\rangle\to\lambda$. Then $\Psi x_n\to 0$, and so 
$\langle(\Phi+\Psi)x_n,x_n\rangle\to\lambda$. Hence $ \lambda\in W_e(\Phi+\Psi)\subset\overline{W(\Phi+\Psi)}$.

Conversely, let $\lambda\in \bigcap\{\overline{W(\Phi+\Psi)}: \Psi:H\to H \hbox{ real-linear of finite rank}\}$.
Let $M\subset H$ be a subspace of finite codimension and $\e>0$. Let $F=M^\perp$. Consider the operator $P_M\Phi|_M:M\to M$, where $P_M$ is the orthogonal projection onto $M$. Fix $\mu\in W(P_M\Phi|_M)$. 

We have $\lambda\in \overline{W(\mu I_F\oplus P_M\Phi|_M)}$ since $\Phi-(\mu I_F\oplus P_M\Phi|_M)$ is a finite-rank real-linear operator. It is easy to see that
$$
W(\mu I_F\oplus P_M\Phi|_M)=\{\mu+t\eta: \eta\in W(P_M\Phi|_M), 0\le t\le 1\}=W(P_M\Phi|_M)
$$
since $W(P_M\Phi|_M)$ is convex. Hence $\lambda\in\overline{W(P_M\Phi|_M)}$ and there exists a unit vector $x\in M$ with $|\langle \Phi x,x\rangle-\lambda|<\e$. By Theorem \ref{T2.2}, $\lambda\in W_e(\Phi)$.
\end{proof}

For antilinear operators we have a better information.

\begin{corollary} Let $A$ be an antilinear operator on an infinite-dimensional Hilbert space $H$. Then $W_e(A)$ is a closed disc centered at $0$.
\end{corollary}

\begin{proof}
By the previous theorem, $W_e(A)$ is a closed convex set. We show that $W_e(A)$ is rotationally symmetric: if $\lambda\in W_e(A)$ and $0\le\theta<2\pi$ then $e^{i\theta}\lambda\in W_e(A)$. Indeed, let $(x_n)$ be a sequence of unit vectors such that $x_n\wto 0$ and $\lambda Ax_n,x_n\rangle\to\lambda$. Let $y_n=e^{-i\theta/2}x_n$ for all $n$. Then $(y_n)$ is weakly null sequence and
$$
\lim_{n\to\infty}\langle Ay_n,y_n\rangle =e^{i\theta}\lim_{n\to\infty}\langle Ax_n,x_n\rangle=e^{i\theta}\lambda.
$$
Hence $e^{i\theta}\lambda\in W_e(A)$.

This implies that $W_e(A)$ is a closed disc centered at $0$.
\end{proof}

The next result is a version of the Lancaster theorem for linear operators, \cite{L}

\begin{theorem}
Let $\Phi$ be a real-linear operator on an infinite-dimensional Hilbert space $H.$ Then 
$$\overline{W(\Phi)} = \operatorname{conv}(W(\Phi) \cup W_{e}(\Phi)).$$
\end{theorem}

\begin{proof}
Clearly $\operatorname{conv}(W(\Phi) \cup W_{e}(\Phi))\subset\overline{W(\Phi)}$ since $\overline{W(\Phi)}$ is convex by Theorem 6.4 \cite{Vmuller2}.

To show the opposite inclusion, suppose on the contrary that there exists $z\in \overline{W(\Phi)}$ such that
$z\notin \operatorname{conv}(W(\Phi) \cup W_{e}(\Phi))$.

After applying a suitable affine transformation, we may assume, without loss of generality, that
$$
z=0
\qquad\text{and}\qquad
\operatorname{Re}W(\Phi)\geq0.
$$
Since $0\in \overline{W(\Phi)}$, there exists a sequence of unit vectors $\{x_n\}$ such that $
\langle \Phi x_n,x_n\rangle\longrightarrow0
~\text{as }n\to\infty.$ 
Passing to a subsequence, we get $x_n\overset{w}\longrightarrow x$ for some $x\in H$ with $\|x\|\leq1$. We consider three cases.

\textbf{(a) ${x=0}$.}
In this case, the sequence of unit vectors ${x_n}$ converges weakly to $0$ and satisfies$
\langle\Phi x_n,x_n\rangle\to 0.$ So $0\in W_e(\Phi).$

\textbf{(b) ${\|x\|=1}$.}
Since $\|x_n\|=1$ and $x_n\overset{w}\longrightarrow x$, we have
$$
\begin{aligned}
\lim_{n\to\infty}\|x-x_n\|^2
&=\lim_{n\to\infty}\bigl(\|x\|^2+\|x_n\|^2
-2\operatorname{Re}\langle x,x_n\rangle\bigr)\\
&=\lim_{n\to\infty}
(2-2\operatorname{Re}\langle x,x\rangle)
=0.
\end{aligned}
$$ Consequently, $x_n\to x$ and
$$
\begin{aligned}
|\langle\Phi x,x\rangle|
&\leq
|\langle\Phi(x-x_n),x\rangle|
+|\langle\Phi x_n,x-x_n\rangle|
+|\langle\Phi x_n,x_n\rangle|\\
&\leq
\|\Phi\|\cdot \|x-x_{n}\|\cdot \|x\|
+\|\Phi\|\cdot\|x_n\|\cdot\|x-x_n\|
+|\langle\Phi x_n,x_n\rangle| \to 0.
\end{aligned}
$$
Hence $\langle\Phi x,x\rangle=0,$ and therefore $0\in W(\Phi).$

\textbf{(c) ${0<\|x\|<1}$.}
In this case,
$$
\begin{aligned}
\lim_{n\to\infty}\|x-x_n\|^2
&=\lim_{n\to\infty}\bigl(\|x\|^2+\|x_n\|^2
-2\operatorname{Re}\langle x,x_n\rangle\bigr)\\
&=\lim_{n\to\infty}
\bigl(\|x\|^2+1-2\operatorname{Re}\langle x,x\rangle\bigr) =1-\|x\|^2>0.
\end{aligned}
$$

Furthermore,

$$
\begin{aligned}
\langle\Phi(x-x_n),x-x_n\rangle
&=
\langle\Phi x,x\rangle
-\langle\Phi x,x_n\rangle
-\langle\Phi x_n,x\rangle
+\langle\Phi x_n,x_n\rangle
\to
-\langle\Phi x,x\rangle.
\end{aligned}
$$
Let us consider $y_{n} = \frac{x-x_{n}}{\|x-x_{n}\|}$ for all $n \in \mathbb{N}.$ Then the sequence of unit vectors $\{y_{n}\}$ converges weakly to $0$ and $\langle \Phi y_{n}, y_{n} \rangle \to  \frac{-\langle \Phi x, x\rangle}{1-\|x\|^{2}}$ as $n \to \infty.$ Now we obtain that 
$$0 = \|x\|^{2} \bigg\langle \phi \bigg(\frac{x}{\|x\|}\bigg),\frac{x}{\|x\|} \bigg\rangle + (1- \|x\|^{2}) \frac{-\langle \Phi x, x\rangle}{1-\|x\|^{2}}$$
is in $\operatorname{conv}(W(\Phi) \cup W_{e}(\Phi)).$ Therefore, $\overline{W(\Phi)} = \operatorname{conv}(W(\Phi) \cup W_{e}(\Phi)).$
\end{proof}

\begin{proposition}
Let $\Phi$ be a real-linear operator on an infinite-dimensional Hilbert space $H.$ Then 
$$
W_e(\Phi^*)=\{\bar z: z\in W_e(\Phi)\}.
$$
\end{proposition}

\begin{proof}
By Theorem \ref{T2.4} and Theorem 5.3 \cite{Vmuller2}, we have
\begin{align*}
W_e(\Phi^*)
&=
\bigcap\{\overline{W(\Phi^*+\Psi)}: \Psi\in\cR(H) \hbox{ compact}\}\\
& = \bigcap\bigl\{\{\bar z: z\in\overline{W(\Phi+\Psi^*)}\}: \Psi\in\cR(H) \hbox{ compact}\bigr\}\\
&=\{\bar z: z\in W_{e}(\Phi)\}.
\end{align*}
\end{proof}

\section{Maximal numerical range}

The maximal numerical range of antilinear operators on Hilbert spaces was introduced in~\cite{amara2025maximal}, where the authors proved that, for a separable Hilbert space \(H\), the maximal numerical range of a bounded antilinear operator is either a circle or a disk centered
at the origin. In the following theorem, we establish the same characterization without imposing the separability assumption on \(H\). Moreover, our proof is more elementary and does not rely on the spectral theorem for the positive self-adjoint operator $\lvert A\rvert $.

\begin{definition} Let $\Phi$ be a bounded real-linear operator on $H$. The \emph{maximal numerical range} of $\Phi$ is the set of all $\lambda\in\bC$ such that there exists a sequence $(x_n)$ of unit vectors in $H$ satisfying $\|\Phi x_n\|\to\|\Phi\|$ and $\langle \Phi x_{n}, x_{n}\rangle \to \lambda$.
\end{definition}

\begin{theorem}\label{thm 3.2}
Let $A$ be a bounded antilinear operator on $H.$ Then $W_{max}(A)$ is either a circle centered at $0$, or a closed disc centered at $0$.
\end{theorem}

\begin{proof} It is immediate to show that the set $W_{max}(A)$ is closed and rotationally symmetric about the origin.  If $A=0$ then $W_{max}(A)=\{0\}$, which is a degenerated disc.

Without loss of generality we may asume that $\|A\| =1.$ Suppose that $W_{max}(A)$ is not a circle. Let $\lambda, \mu \in W_{max}(A)$ satisfy $|\lambda|\ne|\mu|$. Then there exist two sequences $\{x_{n}\}, \{y_{n}\}$ is $H$ with
$$\|x_{n}\| =1 =\|y_{n}\|, ~  \|Ax_{n}\| \to 1, \|Ay_{n}\| \to 1 \text { and } \langle Ax_{n}, x_{n} \rangle \to \lambda, \langle Ay_{n}, y_{n}\rangle \to \mu,$$ as $n \to \infty$. 
We show that $t\lambda + (1-t)\mu \in W_{max}(A),$ for all $t \in (0, 1).$ 

Consider the closed subspaces $M_{n} = \bigvee\{x_{n},y_{n}\}$ and the orthogonal projection $P_{n}$ onto the closed subspace $M_{n},$ for all $n \in \mathbb{N}.$ 
Since $|\lambda|\ne|\mu|$, the vectors $x_n,y_n$ are linearly independent for all $n$ large enough. Thus $\dim M_n=2$. Since the numerical range of an antilinear operator on an at least two-dimensional Hilbert space is convex, then we obtain $$t\langle Ax_{n}, x_{n}\rangle + (1-t)\langle Ay_{n}, y_{n}\rangle \in W(P_{n}A\vert_{M_{n}}),$$ for all $n\in \mathbb{N}.$
Hence, there exists  a sequence $\{u_{n}\}$ in $H$ where $u_{n} = \alpha_{n}x_{n} + \beta_{n}y_{n} \in M_{n}$ with $\alpha_{n}, \beta_{n} \in \mathbb{C}$ and  $\|u_{n}\|=1,$ for all $n \in \mathbb{N}$ such that $$\langle Au_{n},u_{n}\rangle = t\langle Ax_{n}, x_{n}\rangle + (1-t)\langle Ay_{n}, y_{n}\rangle\to t\lambda+(1-t)\mu.$$

Without loss of generality we may assume that the sequence $(\langle x_n,y_n\rangle)_n$ is convergent, $\langle x_n,y_n\rangle\to\eta$ for some $\eta$, $|\eta|\le 1$.
If $|\eta|=1$ then
$$
\|x_n-\eta y_n\|^2=
\|x_n\|^2+\|y_n\|^2-2\Re \langle x_n,\eta y_n\rangle=2-2\Re\bar\eta\langle x_n,y_n\rangle\to 0
$$
as $n\to\infty$. Thus $x_n-\eta y_n\to 0$ and $A(x_n-\eta y_n)\to 0$. We have
$$
\lambda-\bar\eta^2\mu=
\lim_{n\to\infty}\Bigl(\langle Ax_n,x_n\rangle-\langle A(\eta y_n),\eta y_n\rangle\Bigr)
$$
$$
=
\lim_{n\to\infty}\Bigl(\langle Ax_n,x_n-\eta y_n\rangle+\langle Ax_n,\eta y_n\rangle +\langle A(\eta y_n),x_n-\eta y_n\rangle -\langle A(\eta y_n),x_n\rangle\Bigr)
$$
$$
=\lim_{n\to\infty}\Bigl(\langle Ax_n,\eta y_n\rangle-\langle A(\eta y_n),x_n\rangle\Bigr).
$$
Thus 
$$
0=\lim_{n\to\infty} \bigl\langle A(x_n-\eta y_n),x_n+\eta y_n\bigr\rangle
$$
$$
=
\lim_{n\to\infty}\Bigl(\langle Ax_n,x_n\rangle-\bar\eta^2\langle Ay_n,y_n\rangle+\langle Ax_n,\eta y_n\rangle-\langle A(\eta y_n),x_n\rangle\Bigr)
=2(\lambda-\bar\eta^2\mu).
$$
Hence $\lambda=\bar\eta^2\mu$ and $|\lambda|=|\mu|$ a contradiction.

Hence we have $|\eta|<1$ and
$$
1= \lim_{n\to\infty}\|\alpha_n x_n+\beta_n y_n\|^2=
\lim_{n\to\infty}\bigl(|\alpha_n|^2+|\beta_n|^2+2\Re(\alpha_n\bar\beta_n\langle x_n,y_n\rangle)\bigr)
$$
$$
\ge
\lim_{n\to\infty}\bigl(|\alpha_n|^2+|\beta_n|^2-2|\alpha_n\beta_n\eta|\bigr)
\ge
\lim_{n\to\infty}
\bigl((1-|\eta|)(|\alpha_n|^2+\beta_n|^2)+|\eta|\cdot(|\alpha_n|-|\beta_n|)^2\bigr).
$$
Clearly, the two inequalities $\limsup_{n\to\infty}\vert \alpha_{n}\vert \le \frac{1}{\sqrt{1-|\eta|}}$ and  $\limsup_{n\to\infty}\vert \beta_{n}\vert \le \frac{1}{\sqrt{1-|\eta}}$ hold. In particular, the sequences $(|\alpha_n|)$ and $(|\beta_n|)$ are bounded.


Since $\lim_{n \to \infty}\|Ax_{n}\| = \lim_{n \to \infty}\|Ay_{n}\| =1= \|A\| = \|A^{*}\|$,
we have 
\begin{align*}
\lim_{n \to \infty} \|(A^*A-I)x_{n}\|^2 &=
\lim_{n\to\infty}\Bigl(\|A^*Ax_n\|^2+\|x_n\|^2-2\Re\langle A^*Ax_n,x_n\rangle\Bigr)\\ &\le 2-2\lim_{n\to\infty}\Re \langle Ax_n,Ax_n\rangle= 0
\end{align*}
Thus $\|(A^*A-I)x_n\|\to 0$, and
similarly, $\lim_{n \to \infty} \|(A^{*}A-I)y_{n}\| = 0.$ 
By using the boundedness of the coefficients $\alpha_{n}$ and $\beta_{n}$, we have $$\lim_{n \to \infty}(A^*A-I)(\alpha_{n}x_{n} + \beta_{n}y_{n}) = 0.
$$ 
Therefore,
$$\|A(\alpha_{n}x_{n} + \beta_{n}y_{n})\|^{2} = 1 + \bigl\langle (A^*A-I)(\alpha_{n}x_{n} + \beta_{n}y_{n}), (\alpha_{n}x_{n} + \beta_{n}y_{n})\bigr\rangle \to 1= \|A\|^{2},
$$ 
as $n \to \infty.$ It proves that $t\lambda + (1-t)\mu =\lim_{n\to\infty}\langle Au_n,u_n\rangle\in W_{max}(A),$ when $|\lambda| \ne |\mu|.$\\
If there are $\lambda,\mu\in W_{max}(A)$ with $|\lambda|\ne|\mu|$ then the rotational symmetry of $W_{max}(A)$ implies that $|\lambda|, -|\mu|\in W_{max}(A)$, and so the segment between $|\lambda|$ and $-|\mu|$ is contained in $W_{max}(A)$. Thus $W_{max}(A)$ contains all the closed disc centered at $0$ with radius 
$\max\{|\lambda|,|\mu|\}$.
Therefore, $W_{max}(A)$ is a closed disc centered at $0.$

\end{proof}

The next example shows that the maximal numerical range of an antilinear operator can be really a circle.

\begin{example}
Let $A:\ell_2\to\ell_2$ be defined by $A(z_1,z_2,\dots)=(\bar z_1,0,0,\dots)$ for all $(z_1,z_2,\dots)\in\ell_2$.
Then $\|A\|=1$. Let $\{x_n\}$ be a sequence of unit vectors, $\|Ax_n\|\to 1$ and $x_n=(x_{n,1},x_{n,2},\dots)$ for all $n\in\bN$. Then $|x_{n,1}|\to 1$ and we have $|\langle Ax_n,x_n\rangle|=|\bar x_{n_1}|^2\to 1$. Hence $W_{max}(A)\subset\{z: |z|=1\}$. Since $W_{max}(A)$ is circularly symmetric, we have $W_{max}(A)=\{z:|z|=1\}$.
\end{example}

\begin{theorem}\label{thm 3.4} Let $A$ be a bounded antilinear operator on $H.$ Then $$W_{max}(A) = W_{max}(A^*).$$
\end{theorem}

\begin{proof} 
We first assume, without loss of generality, that \(\|A\|=1\).
Let \(\lambda\in W_{max}(A)\). By the definition of the maximal
numerical range, there exists a sequence \(\{z_n\}\) in $H$ such that
\[
\|z_n\|=1,\qquad
\|Az_n\|\to 1,
\]
and
\[
\langle Az_n,z_n\rangle\to\lambda
\qquad\text{as }n\to\infty.
\]
For each \(n\) large enough $Az_n\ne 0$ and we can define
\[
x_n=\frac{Az_n}{\|Az_n\|}.
\]
Then \(\|x_n\|=1\). Moreover, by the definition of the adjoint
\(A\),
\[
\langle A^*x_n,z_n\rangle
=
\frac{\langle A^*Az_n,z_n\rangle}{\|Az_n\|}=
\frac{\|Az_n\|^2}{\|Az_n\|}
=
\|Az_n\|
\to 1.
\]
Since \(\|A^*x_n\|\leq\|A^*\|=1\), it follows that
\[
\|A^*x_n\|\to 1.
\]
Consequently,
$\|A^*x_n-z_n\|^2
=
\|A^*x_n\|^2+\|z_n\|^2
-2\operatorname{Re}\langle A^*x_n,z_n\rangle
\to 0.$
Finally, using $x_n=\frac{Az_n}{\|Az_n\|},$
we obtain 
$$
\begin{aligned}
\langle A^*x_n,x_n\rangle
&=
\frac{\langle A^*x_n,Az_n\rangle}
     {\|Az_n\|}  \\
&=
\frac{\langle A^*x_n-z_n,Az_n\rangle}
     {\|Az_n\|}
+
\frac{\langle z_n,Az_n\rangle}
     {\|Az_n\|}.
\end{aligned}
$$
Hence
$\langle A^*x_n,x_n\rangle\to \overline\lambda$ and $\bar\lambda\in W_{max}(A^*)$. 
The rotationally symmetric shape of $W_{max}(A^*)$ shows that
\(\lambda\in W_{max}(A^*)\). Therefore,
$W_{max}(A)\subseteq W_{max}(A^*).$ By symmetry of the adjoint we get the equality $W_{max}(A)=W_{max}(A^*).$


\end{proof}

\begin{theorem} \label{thm 3.5}
Let \(K\) be a nonzero compact antilinear operator on \(H\). Then $$W_{max}(K)
=
W\!\left(
PK|_{{N(\|K\|^2-K^*K)}}
\right)$$
where $P$ is the orthogonal projection onto $N(\|K\|^2-K^*K)$.
\end{theorem}

\begin{proof}
It is immediate that $K^*K$ is also a compact and selfadjoint linear operator. So, $\|K\|^{2}$ is an isolated point of the spectrum of $K^*K.$ The space $M:=N(\|K\|^{2}- K^*K)$ is a non-trivial finite-dimensional subspace. Then we get a positive real number $c < \|K\|^{2}$ such that $\langle K^*K v, v \rangle \le c \|v\|^{2},$ for all $v \in M^{\perp}.$

Let $\lambda\in W_{max}(K)$.  Then there exists a sequence $(x_n)$ of unit vectors in $H$ such that $\|Kx_n\|\to \|K\|$ and $\langle Kx_n,x_n\rangle\to\lambda$. Write $x_n=u_n+v_n$ with $u_n\in M$ and $v_n\perp M$. We have
$$
\|Kx_n\|^2=\|Ku_n\|^2+\|Kv_n\|^2\le \|K\|^2-(\|K\|^2-c)\|v_n\|^2.
$$
Thus $\|v_n\|\to 0$ and $\|u_n\|\to 1$. Since $(u_n)$ is a bounded sequence in the finite-dimensional subspace $M$, we may assume without loss of generality that $(u_n)$ is convergent, $u_n\to u\in M$. Clearly $\|u\|=1$.
Moreover,
$$
\langle Ku,u\rangle =\lim_{n\to\infty}\langle Ku_n,u_n\rangle=
\lim_{n\to\infty}\langle Kx_n,x_n\rangle=\lambda.
$$
Hence $W_{max}(K)\subset W(PK|_M)$.

Conversely, if $\lambda\in W(PK|_M)$ then there exists a unit vector $x\in M$ with $\langle Kx,x\rangle=\lambda$. Considering the constant sequence $(x)$ gives $\lambda\in W_{max}(K)$. Hence $W_{max}(K)= W(PK|_M)$.
\end{proof}

\section{Essential maximal numerical range}
The essential maximal numerical range for bounded linear operators on Hilbert spaces was introduced in \cite{fong1979essential}. Inspired by this notion, we naturally extend the concept to the setting of antilinear operators and introduce the corresponding essential maximal numerical range as follows.

\begin{definition} Let $\Phi$ be a real-linear operator on an infinite-dimensional Hilbert space $H.$ The essential maximal numerical range of $\Phi$ is the set of all $\lambda\in\bC$ such that there exists a sequence $(x_n)$ of unit vectors in $H$ satisfying $x_n\wto 0$, $\|\Phi x_n\|\to\|\Phi\|_e$ and $\langle \Phi x_n,x_n\rangle\to\lambda$.


\end{definition}

\begin{theorem}\label{thm 4.2}
Let $\Phi$ be a bounded real-linear operator on $H.$ Then $W_{e,max}(\Phi)$ is a compact convex subset of $\bC$.
\end{theorem}

\begin{proof}It is immediate from the definition that \({W}_{e,max}(\Phi)\) is closed and bounded. Hence $W_{e,max}(\Phi)$ is a compact set.

We now prove that \({W}_{e,max}(\Phi)\) is
convex. Without loss of generality, assume that
\(\|\Phi\|_{\mathrm e}=1\).

Let \(\lambda,\mu\in {W}_{e,max}(\Phi)\). Then there exist sequences
\(\{x_n\}\) and \(\{y_n\}\) of unit vectors in \(H\) such that
\[
x_n \overset{w}{\to} 0,\qquad y_n \overset{w}{\to} 0,
\]
\[
\|\Phi x_n\|\to 1,\qquad
\|\Phi y_n\|\to 1,
\]
and
\[
\langle \Phi x_n,x_n\rangle\to\lambda,\qquad
\langle \Phi y_n,y_n\rangle\to\mu
\]
as \(n\to\infty\).

Write $\Phi=T+A$ with $T$ a linear and $A$ antilinear operator. 
Since $x_{n}\wto 0$ and \(y_n\wto 0\), by passing to
a suitable subsequence of \(\{x_n\}\), and relabeling if necessary,
we may assume that
\[
\left|\langle x_n,y_n\rangle\right|<\frac1n,\qquad
\left|\langle \Phi y_n,x_n\rangle\right|<\frac1n,
\]
\[
\left|\langle x_n, T^*y_n\rangle\right|<\frac1n,
\qquad
\left|\langle x_n, A^*y_n\rangle\right|<\frac1n,
\]
\[
\left|\langle  x_n,T^*\Phi y_n\rangle\right|<\frac1n,\qquad
\left|\langle x_n, A^*\Phi y_n\rangle\right|<\frac1n,
\]
for every \(n\in\mathbb N\).
In particular we have
$$
\lim_{n\to\infty}\langle \Phi x_n,y_n\rangle=\lim_{n\to\infty}\langle (T+A)x_n,y_n\rangle=
\lim_{n\to\infty}\langle x_n, T^*y_n\rangle+\lim_{n\to\infty}\overline{\langle x_n, A^*y_n\rangle}=0
$$
and similarly,
$$
\lim_{n\to\infty}\langle\Phi x_n,\Phi y_n\rangle=0.
$$

Fix \(t\in[0,1]\), and put $z_n=\sqrt{t}\,x_n+\sqrt{1-t}\,y_n.$
Then $$\|z_n\|^2
=
t\|x_n\|^2+(1-t)\|y_n\|^2
+2\sqrt{t(1-t)}
\operatorname{Re}\langle x_n,y_n\rangle,$$
and hence $\|z_n\|\to 1$ as $n \to \infty.$
For $n$ large enough define
$w_n=\frac{z_n}{\|z_n\|}.$
Then \(\|w_n\|=1\) and \(w_n \wto 0\).

Since \(\Phi\) is real-linear and \(t\in[0,1]\) is real, we have
\[
\Phi z_n=\sqrt{t}\,\Phi x_n+\sqrt{1-t}\,\Phi y_n.
\]
Consequently,
\[
\begin{aligned}
\|\Phi z_n\|^2
={}&t\|\Phi x_n\|^2+(1-t)\|\Phi y_n\|^2\\
&+2\sqrt{t(1-t)}
\operatorname{Re}\langle \Phi x_n,\Phi y_n\rangle\to 1
\end{aligned}
\]
as $n\to\infty$. 
It follows that $
\|\Phi w_n\|^2
=
\frac{\|\Phi z_n\|^2}{\|z_n\|^2}
\to 1$ as $n \to \infty.$

Moreover,
\[
\begin{aligned}
\langle \Phi w_n,w_n\rangle
={}&\frac{1}{\|z_n\|^2}
\Bigl(
t\langle \Phi x_n,x_n\rangle
+(1-t)\langle \Phi y_n,y_n\rangle\\
&\qquad
+\sqrt{t(1-t)}
\bigl(
\langle \Phi x_n,y_n\rangle
+\langle \Phi y_n,x_n\rangle
\bigr)
\Bigr) \to
t\lambda+(1-t)\mu,
\end{aligned}
\] as $n \to \infty.$

Thus $t\lambda+(1-t)\mu\in {W}_{e,max}(\Phi )$ for all $t\in[0,1]$,  
and hence \({W}_{e,max}(\Phi)\) is convex.
\end{proof}

The next proposition follows immediately from the definition.

\begin{proposition}
Let $\Phi$ be a real-linear operator on an infinite-dimensional Hilbert space $H$. Then:
\begin{itemize}
\item[(i)] $W_{e,max}(\Phi)\subset W_e(\Phi)$;
\item[(ii)] $W_{e,max}(\Phi+\Psi)=W_{e,max}(\Phi)$ for all compact real-linear operators $\Psi\in\cR(H)$.
\end{itemize}
\end{proposition}

For antilinear operators we have again a better information.

\begin{corollary}
Let $A$ be an antilinear operator on an infinite-dimensional Hilbert space $H$. Then $W_{e,max}(A)$ is a closed disc centered at the origin.
\end{corollary}

\begin{proof}
It is easy to see from the definition that $W_{e,max}(A)$ is rotationally symmetric. By the previous theorem it is a convex compact set. Hence $W_{e,max}(A)$ is a closed disc centered at the origin (possibly degenerated to the single point $\{0\}$).

\end{proof}
 Theorem \ref{thm 3.2} shows that the maximal numerical range of a real-linear operator need not be convex in general. This naturally raises the question of whether it is nevertheless connected. The following theorem provides an affirmative answer to this question, establishing that the maximal numerical range of a real-linear operator is connected.

 \begin{theorem}\label{thm 4.4.0}
Let $\Phi$ be a real-linear on a Hilbert space $H.$ Then $W_{max}(\Phi)$ is a connected subset of $\mathbb{C}.$
 \end{theorem}

 \begin{proof}
We distinguish two cases.
\smallskip

\textbf{Case I:} Let $\|\Phi\|_{e} = \|\Phi\|.$ We show that in this case $W_{max}(\Phi)$ is even star-shaped where any element of $W_{e,max}(\Phi)$ is a star-center of $W_{max}(\Phi)$.
Let $\lambda\in W_{e,\max}(\Phi)$ and $\mu\in W_{\max}(\Phi)$.
Let $(x_n)$ and $(y_n)$ be sequences of unit vectors in $H$ such that $x_n\wto 0$, $\|\Phi x_n\|\to \|\Phi\|_e=\|\Phi\|$, $\|\Phi y_n\|\to\|\Phi\|$, $\langle \Phi x_n,x_n\rangle\to\lambda$ and $\langle \Phi y_n,y_n\rangle\to\mu$.
By an argument analogous to the proof of Theorem~\ref{thm 4.2}, the line segment joining $\lambda$ and $\mu$ is contained in $W_{max}(\Phi)$. Hence $\lambda$ is a star-center of $W_{max}(\Phi)$. In particular, $W_{max}(\Phi)$ is connected.


\smallskip

 \textbf{Case II:} ~Let $\|\Phi\|_{e} < \|\Phi\|.$ Without loss of generality we may assume that $\|\Phi\|=1$.
Let $\lambda\in W_{max}(\Phi)$. Let $(x_n)$ be a sequence of unit vectors in $H$ such that $\|\Phi x_n\|\to 1$ and $\langle \Phi x_n,x_n\rangle\to\lambda$.

By passing to a subsequence if necessary we may assume that the sequence $(x_n)$ is weakly convergent, $x_n\wto x$ for some $x\in H$, $\|x\|\le 1$. We show that $\|x\|=1$. Suppose on the contrary that $\|x\|<1$.
We have $x_n-x\wto 0$ and $\Phi(x_n-x)\wto 0$. We have
$$
1=\lim_{n\to\infty}\|x_n\|^2
=
\lim_{n\to\infty}\Bigl(\|x\|^2+\|x_n-x\|^2+2\Re\langle x,x_n-x\rangle\Bigr)
=\|x\|^2+\lim_{n\to\infty}\|x_n-x\|^2.
$$
Similarly,
$$
1=\lim_{n\to\infty}\|\Phi x_n\|^2
=
\lim_{n\to\infty}\Bigl(\|\Phi x\|^2+\|\Phi(x_n-x)\|^2+2\Re\langle \Phi x,\Phi(x_n-x)\rangle\Bigr)
$$
$$
=\|\Phi x\|^2+\lim_{n\to\infty}\|\Phi(x_n-x)\|^2.
$$
Let $y_n=\frac{x_n-x}{\|x_n-x\|}$. Then $\|y_n\|=1$, $y_n\wto 0$ and
$$
\lim_{n\to\infty}\|\Phi y_n\|^2=
\lim_{n\to\infty}\frac{\|\Phi(x_n-x)\|^2}{\|x_n-x\|^2}=
\frac{1-\|\Phi x\|^2}{1-\|x\|^2}\ge 1,
$$
since $\|\Phi x\|^2\le\|x\|^2$. Hence $\|\Phi\|^2_e\ge\lim_{n\to\infty}\|\Phi y_n\|^2\ge 1=\|\Phi\|^2$ and $\|\Phi\|_e=\|\Phi\|=1$, a contradiction. Thus $\|x\|=1$, which implies that $x_n\to x$ in the norm.

Consequently, $\|\Phi x\|=\lim_{n\to\infty}\|\Phi x_n\|=\|\Phi\|=1$ and
$$
\lambda=\lim_{n\to\infty}\langle\Phi x_n,x_n\rangle=\langle\Phi x,x\rangle.
$$
Let 
$$
\mathcal{M}(\Phi)=\{\langle\Phi x,x\rangle: x\in H, \|x\|=1, \|\Phi x\|=1\}.
$$
We have proved $W_{max}(\Phi)\subset \mathcal{M}(\Phi)$. Since the second inclusion is clear, we have $W_{max}(\Phi)= \mathcal{M}(\Phi)$.

Let $x$ be a unit vector and $\|\Phi\|=1$. The following conditions are equivalent:
$$
\|\Phi x\|=1,
$$
$$
\Re\langle\Phi x,\Phi x\rangle=1,
$$
$$
\Re\langle\Phi^*\Phi x,x\rangle=1,
$$
$$
\Phi^*\Phi x=x,
$$
$$
x\in N(I-\Phi^*\Phi).
$$
Hence 
$$
W_{max}(\Phi)=\{\langle \Phi x,x\rangle: x\in N(I-\Phi^*\Phi), \|x\|=1\}.
$$
The mapping $x\mapsto \langle\Phi x,x\rangle $ is continuous, and so $W_{max}(\Phi)$ is a continuous image of the unit sphere in $N(I-\Phi^*\Phi)$. Hence $W_{max}(\Phi)$ is connected.

 \end{proof}

\begin{example}
Let $\Phi$ be the real-linear operator on $\ell^{2}$ defined by

$$
\Phi(x_{1},x_{2},\ldots)
=
\bigl(\operatorname{Re}x_{1},0,0,\ldots\bigr).
$$

We claim that $W_{\max}(\Phi)=\{1\}.$

Indeed, let $\lambda\in W_{\max}(\Phi)$. By the definition of the maximal numerical range, there exists a sequence of unit vectors ${z_n}\subset\ell^{2}$ such that

$$
\|\Phi z_n\|\to \|\Phi\|=1
$$

and

$$
\langle\Phi z_n,z_n\rangle\to\lambda.
$$

Writing $z_n=(z_{n,1},z_{n,2},\ldots),$ we obtain $\|\Phi z_n\|^2
=
\bigl|\operatorname{Re}z_{n,1}\bigr|^2
\to 1.$

Since $\|z_n\|=1$, we have $
\bigl|\operatorname{Re}z_{n,1}\bigr|
\leq |z_{n,1}|
\leq 1.
$

Therefore, $
|z_{n,1}|\to 1.
$

It follows that

$$
|\operatorname{Im}z_{n,1}|^2
=
|z_{n,1}|^2-\bigl(\operatorname{Re}z_{n,1}\bigr)^2
\to 0.
$$

Moreover,

$$
\langle\Phi z_n,z_n\rangle
=
\operatorname{Re}z_{n,1}\overline{z_{n,1}}
=
\bigl(\operatorname{Re}z_{n,1}\bigr)^2
-i\operatorname{Re}z_{n,1}\operatorname{Im}z_{n,1}.
$$
Since $
\bigl(\operatorname{Re}z_{n,1}\bigr)^2\to 1
$ and $
\operatorname{Re}z_{n,1}\operatorname{Im}z_{n,1}\to 0,
$ we conclude that $
\langle\Phi z_n,z_n\rangle\to 1.$ Hence $\lambda=1$, and consequently ${W_{\max}(\Phi)=\{1\}.}$

\end{example}

\begin{remark}
It is straightforward to observe that Theorem \ref{thm 3.5} can be generalized to

$$
W_{max}(A) = W\big( PA|_{N(A^*A - \|A\|^{2})}\big),
$$
where \(A\) is an antilinear operator on \(H\), \(\|A\|_{e} < \|A\|\) and $P$ is the orthogonal projection onto $N(A^{*}A - \|A\|^{2})$. This follows from Case 2 of Theorem \ref{thm 4.4.0}, which states that \(W_{max}(A) = \mathcal{M}(A)\). Moreover, it is readily verified that

$$
\mathcal{M}(A) = W\big( PA|_{N(A^*A - \|A\|^{2})}\big).
$$

\end{remark}

\begin{theorem}
Let $A$ be a bounded antilinear operator on $H.$ Then $$W_{e, max}(A) = W_{e,max}(A^*).$$
\end{theorem}

\begin{proof}
The proof follows the same argument as that of Theorem \ref{thm 3.4}.
\end{proof}

\section{$q$-numerical range}

\begin{definition}
Let $\Phi$ be a real-linear operator on a Hilbert space $H$. Let $q\in\bC$, $|q|\le 1$. The $q$-{\it numerical range} 
of $\Phi$ is the set
$$
W(\Phi:q)=\{\langle \Phi x,y\rangle: \|x\|=\|y\|=1, \langle x,y\rangle=q\}.
$$
\end{definition}

\begin{proposition}
Let $\dim H\ge 3$, $|q|\le 1$, let $A:H\to H$ be antilinear. Then $0\in W(A:q)$.
\end{proposition}

\begin{proof}
There exists a unit vector $x\in H$ such that $\langle Ax,x\rangle=0$. Since $\dim H\ge 3$, there exists a unit vector $u\in H$, $u\perp x,Ax$.

Set $y=\bar q x+\sqrt{1-|q|^2} u$. We have $\|y\|=1$,
$$
\langle x,y\rangle=q
$$
and
$$
\langle Ax,y\rangle=q\langle Ax,x\rangle+\sqrt{1-|q|^2}\langle Ax,u\rangle=0.
$$
Hence $0\in W(A:q)$.
\end{proof}

For two-dimensional spaces this is not true.

\begin{example}
Let $\dim H=2$, $H=\bC^2$. Let $A:H\to H$ be defined by $A(z_1,z_2)=(\bar z_2,-\bar z_1)$.
Clearly $A$ is an antilinear operator. It is easy to verify that $\|Ax\|=\|x\|$ and $Ax\perp x$ for all $x\in H$.

Let $|q|<1$. We show that $0\notin W(A:q)$. Suppose on the contrary that there exist unit vectors $x,y\in H$ with
$\langle x,y\rangle=q$ and $\langle Ax,y\rangle=0$. The pair $(x,Ax)$ is an orthonormal basis in $H$. Thus
$$
1=\|y\|^2=|\langle x,y\rangle|^2+|\langle Ax,y\rangle|^2=|q|^2<1,
$$
a contradiction.
Hence $0\notin W(A:q)$.
\end{example}

If $|q|=1$ then $0\in W(A:q)$ for all antilinear operators $A$ on a two-dimensional Hilbert space $H$. This is known for $q=1$: there exists a unit vector $x\in H$ such that $\langle Ax,x\rangle=0$.

If $|q|=1$ then set $y=\bar q x$. Then $\langle x,y\rangle=q$ and $\langle Ax,y\rangle =0$. Hence $0\in W(A:q)$.

For later use we need the following lemma.

\begin{lemma}\label{lemma 6.4}
Let $H$ be a Hilbert space, $k\in\bN$, and let $\{x_1,\dots,x_k\}$ and $\{y_1,\dots,y_k\}$ be two orthonormal sets in $H$. Then there exists a continuous function $f:[0,1]\to H^k$ such that $f(0)=(x_1,\dots,x_k)$, $f(1)=(y_1,\dots,y_k)$ and, for all $t\in[0,1]$,  $f(t)$ is an orthonormal set in $H$.
\end{lemma}

\begin{proof}
Let $F_1=\bigvee_{j=1}^k x_j$ and $F_2=\bigvee_{j=1}^k y_j$. Clearly $\dim F_1^\perp=\dim F_2^\perp$. Let $V:F_1^\perp\to F_2^\perp$ be a linear surjective isometry. Define a linear operator $U:H\to H$ by $Ux_j=y_j\quad(j=1,
\dots,k)$ and $U|_{F_1^\perp}=V$. Clearly $U$ is a unitary operator.
By the spectral theorem for unitary matrices, there exists a self-adjoint linear operator $T$ such that $U = e^{iT}.$\\
Consider the set of unitary operators $\{U(t) = e^{itT}: 0 \leq t \leq 1\}.$ 
Define function $f:[0,1]\to H^k$ by  $f(t)=(U(t)x_1,\dots,U(t)x_k)$. Clearly $f$ satisfies all the required conditions.

\end{proof}

\bigskip 
Let $A:H\to H$ be an antilinear operator, $|q|\le 1$. Let
$$
\Lambda=\bigl\{(x,y)\in H\times H, \|x\|=\|y\|=1, \langle x,y\rangle=q\bigr\}.
$$

\begin{proposition}\label{pro 5.4}
The set $\Lambda$ is connected.
\end{proposition}

\begin{proof}
Let $(x_1,y_1), (x_2,y_2)\in\Lambda$. We show that there exists a continuous function $h:[0,1]\to \Lambda$ such that
$h(0)=(x_1,y_1)$ and $h(1)=(x_2,y_2)$.
\smallskip

If $|q|=1$ then $y_1= \bar q x_1$ and $y_2=\bar q x_2$.
Let $f:[0,1]\to H$ be a continuous function such that $f(0)=x_1$, $f(1)=x_2$ and $\|f(t)\|=1$ for all $t\in[0,1]$. 
Define function $h:[0,1]\to\Lambda $ by
$$
h(t)=(f(t), \bar q f(t)).
$$
Then $h$ is a continuous function, $h(t)\in\Lambda$ for all $t\in [0,1]$, $h(0)=(x_1,y_1)$ and $h(1)=(x_2,y_2)$.
\smallskip

Let $|q|<1$. 
Let $v_1$ be the unit vector in $\bigvee\{x_1,y_1\}$, $v_1\perp x_1$ and $y_1=\bar q x_1+\sqrt{1-|q|^2}v_1$. Similarly, 
let $v_2$ be the unit vector in $\bigvee\{x_2,y_2\}$, $v_2\perp x_2$ and $y_2=\bar q x_2+\sqrt{1-|q|^2}v_2$.\\
By the previous lemma, there exists a continuous function $f:[0,1]\to H^2$ such that $f(0)=(x_1,v_1)$, $f(1)=(x_2,v_2)$ and $f(t)=(x_t,v_t)$ is an orthonormal pair for all $t\in[0,1]$.
Then $(x_t,\bar q x_t+\sqrt{1-|q|^2}v_t)\in\Lambda$ for all $t\in[0,1]$ and $t\mapsto (x_t,\bar q x_t+\sqrt{1-|q|^2}v_t)$ is a continuous function connecting the pairs $(x_1,y_1)$ and $(x_2,y_2)$ in $\Lambda$.
\end{proof}

\begin{theorem}
Let $\dim H\ge 3$, $|q|\le 1$ and let $A:H\to H$ be an antilinear operator. Then $W(A:q)$ is a disc centered at $0$ (either open or closed, possibly degenerated to $\{0\}$).
\end{theorem}

\begin{proof}
The set $W(A:q)$ is clearly rotationally symmetric, connected by the previous theorem and $0\in W(A:q)$ by Proposition \ref{pro 5.4}. Hence $W(A:q)$ is a closed disc centered at $0$.
\end{proof}

\begin{remark}
Assume that $\dim H=2$, $\Phi\in\cR(H)$ and $|q|\le 1$.
Clearly $W(A:q)$ is rotationally symmetric.
The mapping $(x,y)\mapsto \langle\Phi x,y\rangle$ is continuous and maps $\Lambda$ onto $W(\Phi:q)$. Thus $W(\Phi:q)$ is connected and compact.
Denote
$$
a=\inf\{|\langle Ax,y\rangle|:(x,y)\in\Lambda\},\qquad
b=\sup\{|\langle Ax,y\rangle|:(x,y)\in\Lambda\}.
$$
Then $W(\Phi:q)$ is the closed annulus bounded by the circles of radii $a$ and $b.$
\end{remark}

\section{Trace numerical range}

\begin{definition} Let $\Phi$ be a real-linear operator on a Hilbert space $H$ an let $k\in\bN$. Define the set $W_{tr,k}(\Phi)$ by
$$W_{tr,k}(\Phi) = \Bigl\{\frac{1}{k} \sum_{j=1}^{i=k}\langle \Phi x_{j}, x_{j}\rangle : \{x_{j}\}_{j=1}^{j=k} \text{ is an orthonormal set in $H$}\Bigr\}.$$
\end{definition}

\begin{theorem} Let $A$ be an antilinear operator in the infinite dimensional Hilbert space $H.$ Then $W_{tr,k}(A)$ is a disc centered at origin, for all $k \in \mathbb{N}.$

\end{theorem}

\begin{proof} It is evident from the definition that $W_{tr,k}(A)$ is always rotationally symmetric. 

We show that $W_{tr,k}(A) $ is connected. Let $\lambda,\mu\in W_{tr,k}(A)$. 
Then there exists an orthonormal sets  $\{x_{0,j}\}_{j=1}^{j=k}$ and $\{x_{1,j}\}_{j=1}^{j=k}$
such that $$\lambda = \frac{1}{k}\sum_{j=1}^{j=k} \langle Ax_{0,j}, x_{0,j}\rangle \text{ and } \mu = \frac{1}{k}\sum_{j=1}^{j=k} \langle Ax_{1,j}, x_{1,j}\rangle.$$
By Lemma \ref{lemma 6.4}, there exits a continuous function $f:[0,1]\to H^k$ such that $f(0)=(x_{0,1},\dots,x_{0,k})$, $f(1)=(x_{1,1},\dots,x_{1,k})$ and $f(t)=(x_{t,1},\dots,x_{t,k})$ is an orthonotmal set for all $t\in[0,1]$. Then $t\mapsto \frac{1}{k} \sum_{j=1}^{i=k}\langle Ax_{t,j}, x_{t,j}\rangle $ is a continuous curve connecting $\lambda$ and $\mu$ in $W_{tr,k}(A)$.

We show that $0\in W_{tr,k}(A)$. Since, $A$ is an antilinear operator, there exists an unit vector $x_{1}$ in $H$ such that $\langle Ax_{1}, x_{1}\rangle = 0.$ 
Let $M_1=x_1^\perp$ and consider the antilinear operator $P_{M_1}A|_{M_1}$. Since $0\in W(P_{M_1}A|_{M_1})$, there exists a unit vector $x_2\in M_1$ with $\langle Ax_2,x_2\rangle=0$. 
Continuing this argument inductively, we obtain an orthonormal set $\{x_{j}\}_{j=1}^{j=k}$ such that $\langle Ax_{j}, x_{j}\rangle = 0,$ for $i= 1, 2, \dots, k.$ Consequently, $0 \in W_{tr,k}(A).$
Therefore $W_{tr,k}(A)$ is a disc centered at origin.

\end{proof}

\begin{theorem}
Let $\phi$ be an real-linear operator in the infinite dimensional Hilbert space $H.$ Then 
$$W_{e}(\Phi) = \bigcap_{k \in \mathbb{N}} \overline{W_{tr, k}(\Phi)}.$$
\end{theorem}
\begin{proof}
Let $\lambda \in W_{e}(\Phi).$ Then there exists an orthonormal sequence $\{x_{n}\}$ in $H$ such that $$\langle \Phi x_{n}, x_{n}\rangle \to \lambda$$ as $n \to \infty.$  Now consider $\lambda_{k} = \frac{1}{n}\sum_{i=1}^{n}\langle \Phi x_{i + k}, x_{i+k}\rangle \in W_{tr, n}(\Phi)$ for all $k \in \mathbb{N}.$ Thus, $$\lim_{k \to \infty} \lambda_{k} = \lambda \in \overline{W_{tr, n}(\Phi)}.$$  Consequently, $W_{e}(\Phi) \subset \bigcap_{k \in \mathbb{N}} \overline{W_{tr, k}(\Phi)}.$ \\

Conversely,  consider $\mu \in  \bigcap_{k \in \mathbb{N}} \overline{W_{tr, k}(\Phi)}.$ 
We show that $\mu\in \overline{W(\Phi+\cF)}$ for all finite-rank real-linear operators $\cF$.

Take a real-linear finite rank operator $\cF = F_{T} +F_{A},$ where $F_{T}$ is the linear part and $F_{A}$ is the antilinear part. Clearly both $F_T$ and $F_A$ are of finite rank.
 
By using Theorem 2.19 \cite{majumdar2026unbounded}, the polar decomposition theorem of antilinear operator gives $$F_{A}(x) = \sum_{i=1}^{m} \mu_{i}\overline{\langle x, \phi_{i}\rangle} \psi_{i},$$ where $\{\phi_{i}\}_{i=1}^{i=m}$ and  $\{\psi_{i}\}_{i=1}^{i=m}$ both are the set of orthonormal vectors and $\{\mu_{i}\}_{i=1}^{i=m}$ is the set of non-zero eigenvalue of $|F_{A}| = (F_{A}^{*}F_{A})^{\frac{1}{2}}.$\\
Using Lemma 3 \cite{tracek}, for chosen $\e>0,$ we get a natural number $N \in \mathbb{N}$ such that $$W_{tr, N}(F_{T}) \subset \Bigl\{z\in\bC:|z|<\frac{\e}{4} \Bigr\}\text { and } W_{tr, N}(F_{A}) \subset \Bigl\{z:|z|<\frac{\e}{4}\Bigr\}.$$
Since $\mu \in \overline{W_{tr, N}(\Phi)}.$ Then there exists $\gamma \in W_{tr, N}(\Phi)$ such that $|\mu - \gamma| < \frac{\e}{2}.$ It follows that 
$$\gamma = \frac{1}{N} \sum_{j=1}^{N}\langle \Phi z_{j}, z_{j}\rangle,$$ where $\{z_{j}\}_{j=1}^{j=N}$ is the orthonormal set in $H.$ The convexity of the numerical range of real-linear operator (see Theorem 6.4 \cite{Vmuller2}) confirms that $$\alpha = \frac{1}{N} \sum_{j=1}^{j=N} \langle (\Phi +\cF)z_{j}, z_{j}\rangle \in W(\Phi + \cF).$$ Hence, 
$$|\alpha - \gamma| = \bigg|\frac{1}{N}\sum_{j=1}^{j=N} \langle \cF z_{j}, z_{j}\rangle \bigg| < \frac{\e}{2}.$$ 
Therefore, $| \mu - \alpha| \le |\mu - \gamma| + |\gamma- \alpha| < \e.$ It shows that $\mu \in \overline{W(\Phi + \cF)}.$  Using Theorem \ref{T2.4}, we get   $\bigcap_{k \in \mathbb{N}} \overline{W_{tr, k}(\Phi)} \subset  W_{e}(\Phi).$ Consequently, $$\bigcap_{k \in \mathbb{N}} \overline{W_{tr, k}(\Phi)} = W_{e}(\Phi).$$
\end{proof}

\section*{Funding}
The authors gratefully acknowledge support from GA~\v{C}R(Grant No.~25-15444K) and RVO:67985840.

\section*{Data Availability}
Data sharing is not applicable to this article as no data sets were generated or analyzed during the current study.

\section*{Declarations}
  The authors confirm that there are no conflicts of interest regarding the publication of this manuscript.


\begin{thebibliography}{10}
\bibitem{amara2026numerical}
Zouheir Amara, Ting Liu, Cun Wang, and Zhu Sen.
  \newblock{Numerical ranges for compact perturbations of antilinear operators}, \newblock{\em  Linear Algebra Appl.}, 739:104-123, 2026.

\bibitem{fong1979essential}
C.K. Fong.
\newblock{On the essential maximal numerical range},
\newblock{\em Acta Sci. Math}, 41:307--315, 1979.


\bibitem{choi2006higher}
Man-Duen Choi, David W. Kribs, and Karol {\.Z}yczkowski.
\newblock{Higher-rank numerical ranges and compression problems},
\newblock{\em  Linear Algebra Appl.}, 418(2-3):828--839, 2006.
  
\bibitem{amara2025maximal}
 Zouheir Amara, Ting Liu, and Sen Zhu.
  \newblock{Maximal numerical ranges of antilinear operators},
  \newblock{\em J. Math. Anal. Appl.}, 556:130116, 2025.
  
\bibitem{L}
J. Lancaster. 
\newblock{The boundary of the numerical range}, 
\newblock{\em Proc. Amer. Math. Soc.}, 49:393–398, 1975.	

\bibitem{Vmuller}
 Damian Kołaczek, Vladimir Müller. 
 \newblock{Numerical ranges of antilinear operators,} \newblock {\em Integral Equ. Oper. Theory} 96(2): 17, 2024.

 \bibitem{Vmuller2}
 Damian Kołaczek, Vladimir Müller. 
 \newblock{Real linear operators and numerical ranges,} \newblock {\em arXiv preprint arXiv:2508.04347}, 2025.

\bibitem{GauWu2021}
Pei Yuan Wu, Hwa-Long Gau.
\newblock{Numerical Ranges of Hilbert Space Operators},
\newblock{\em Cambridge University Press}, Vol. 179, 2021.

\bibitem{BonsallDuncan1973}
F. F. Bonsall,  J. Duncan.
\newblock{Numerical Ranges II},
\newblock{\em Cambridge University Press}, Vol. 10, 1973.


\bibitem{FillmoreStampfliWilliams1972}
P. A. Fillmore, J. G. Stampfli, and J. P. Williams.
\newblock{On the essential numerical range, the essential spectrum, and a problem of Halmos},
\newblock{\em Acta Sci. Math. (Szeged)}, 33:179--192, 1972.


\bibitem{stampfli1970norm}
 Joseph Stampfli.
  \newblock{The norm of a derivation},
 \newblock{\em Pac. J. Math.}, 33(3):737--747, 1970.
  
\bibitem{chien2002davis}
Mao-Ting Chien, Hiroshi Nakazato.
  \newblock{Davis--Wielandt shell and q-numerical range},
  \newblock{\em Linear Algebra Appl.}, 340:(1-3):15--31, 2002.
  
\bibitem{chien2005q}
Mao-Ting Chien, Hiroshi Nakazato.
  \newblock{The q-numerical ranges of normal operators},
  \newblock{\em Linear Multilinear Algebra }, 53(6):393--416, 2005.

\bibitem{majumdar2026unbounded}
Arup Majumdar.
  \newblock{Unbounded Antilinear Operators on Hilbert Spaces}, \newblock{\em arXiv preprint arXiv:2605.23735}, 2026.

\bibitem{hur2021numerical}
Injo Hur, Ji Eun Lee.
  \newblock{Numerical ranges of conjugations and antilinear operators},
  \newblock{\em Linear Multilinear Algebra}, 69(16):2990--2997, 2021.
  
\bibitem{tracek}
M. S. Jones, B. Nashvadian-Bakhsh and H. P. Rogosinski.
\newblock{Intersection Properties for Generalized
Numerical Ranges}, \newblock{\em Linear Multilinear Algebra}, 34:377-383, 1993.
\end{thebibliography}
\end{document}